\documentclass[final]{aupl}

\usepackage{
amsfonts,
latexsym,
amssymb,
enumerate,
verbatim,
mathrsfs,
}

\newcommand{\labbel}{\label} \newcommand{\bibbitem}{\bibitem}

\newtheorem{theorem}{Theorem}[section]
\newtheorem{lemma}[theorem]{Lemma}

\newtheorem{proposition}[theorem]{Proposition} 
 
\newtheorem{corollary}[theorem]{Corollary}

\newtheorem*{theorem*}{Theorem}
\newtheorem*{corollary*}{Corollary}

\theoremstyle{definition}

\newtheorem{problem}[theorem]{Problem}

\theoremstyle{remark}
\newtheorem{remark}[theorem]{Remark}

\newtheorem*{acknowledgement}{Acknowledgement}

\newcommand{\alt}[1]{^{\circ #1}} 

\newcommand{\astt}{^{\circledast}}

 \allowbreak

 \allowdisplaybreaks[1]

\begin{document}
 
\title{Unions of admissible relations}

\author{Paolo Lipparini} 
\address{Dipartimento Unito di Matematica\\Viale della  Ricerca
 Scientifica\\Universit\`a di Roma ``Tor Vergata'' 
\\I-00133 ROME ITALY}
\urladdr{http://www.mat.uniroma2.it/\textasciitilde lipparin}
\email{lipparin@axp.mat.uniroma2.it}



\thanks{Work performed under the auspices of G.N.S.A.G.A. Work 
partially supported by PRIN 2012 ``Logica, Modelli e Insiemi''}

\begin{abstract}
We show that a variety $\mathcal V$ is congruence distributive if and only if 
there is some $h$ such that the inclusion
\begin{equation} \labbel{1}     
 \Theta \cap ( \sigma \circ \sigma ) \subseteq ( \Theta \cap  \sigma ) \circ
( \Theta \cap  \sigma )  \circ \dots \quad \text{ ($h$ factors)}  
  \end{equation}
holds in every algebra in $\mathcal V$, for every 
tolerance $\Theta$  and every
U-admissible relation
$\sigma$. 
By a \emph{U-admissible} 
relation we mean a binary relation which is
the  set-theoretical union 
of a set of reflexive and admissible relations.
For any fixed $h$, a Maltsev-type characterization is given for the
inclusion \eqref{1}.
It is an open problem whether \eqref{1} 
is still equivalent to congruence distributivity when $\Theta$ 
is assumed to be a $U$-admissible relation, rather than a tolerance.
In both cases many equivalent formulations for \eqref{1}
are presented. 
The results suggest that it might be interesting to study
the structure of the set of U-admissible relations on an algebra,
as well as identities dealing with such relations.
\end{abstract}

\maketitle

\section{Introduction} \labbel{intro}

Congruence identities have played an important role
in universal algebra right from the beginning.
The study of congruence permutable, distributive and modular varieties
has led to a plethora of significant techniques and results.
J{\'o}nsson \cite{cv}  is a good introduction to the classical results.
Research in the field is still very active;
recent, advanced and sophisticated results appear, for example, in 
Kearnes and Kiss \cite{KK}, where the reader can also find further  references.

 Already from  the earlier arguments it appeared evident
that tolerances and, more generally,
 admissible relations are  fundamental tools even
when the main focus are congruences.
See, e.g.,  Gumm \cite{G}, J{\'o}nsson  \cite[p. 370]{cv} or 
 Tschantz \cite{T}. In a form or another
and more or less explicitly, this aspect appears 
also in 
\cite{CH,JD,KK,contol,JDS,ricm,W},
just to limit ourselves to a few references.
Recall that a \emph{tolerance}
is a  reflexive, symmetric and admissible relation.

In \cite{ricm} we have found 
a particularly simple characterization of congruence modularity.
A variety $\mathcal V$ is congruence modular if and only if
there is some $k$ such that   the inclusion
\begin{equation}\labbel{2}     
 \Theta  (R \circ R) \subseteq ( \Theta  R) \alt k
  \end{equation}
holds in every algebra in $\mathcal V$, for every
  reflexive and admissible relation $R$  and 
every congruence (equivalently, every tolerance)
 $\Theta$.

Here and below juxtaposition denotes intersection and, for a binary relation $R$,
$R \alt h$  denotes
the relational composition
 $R \circ R \circ \dots$ with
$h$ factors, that is, with $h-1$ occurrences 
of $\circ$. 
The displayed formula in the abstract is  then written as
$\Theta( \sigma \circ \sigma ) \subseteq (\Theta \sigma) \alt h $. 
We shall also use the shorthand  $S \circ _m T$ to  denote
$ S \circ T \circ S \dots$
with $m$ factors.
Thus
$R \alt h$ is the same as $R \circ _h R$.

We say that an inclusion like \eqref{2} 
\emph{holds} in some variety $\mathcal V$ if
the inclusion holds 
when interpreted in the standard way
in the set of binary relations
on every algebra in $\mathcal V$.
 Similar expressions like
 ``$\mathcal V$ \emph{satisfies} an inclusion'' shall be used
with the same meaning.
Notice that, since an inclusion $A \subseteq B$ is equivalent to the identity
$A=AB$, we can equivalently---and usually shall---speak of \emph{identities}.
As a standard convention, we assume that 
juxtaposition ties more than any other binary operation symbol, namely,
an expression like $(\alpha \beta ) \circ (\alpha \gamma )$
will be simply written as 
$\alpha \beta \circ \alpha \gamma $. 
However, exponents bind more than anything else,
e.g.,  $\Theta R \alt h$ means
$\Theta (R \alt h)$.
If not explicitly stated otherwise,
$\alpha$, $\beta$ and $\gamma$ 
are variables for congruences,
$\Theta$ can be equivalently taken to be a variable for
congruences or tolerances and 
$R$, $S$ and $T$ are variables for reflexive and admissible relations.
All the binary relations considered in this note are assumed to be 
reflexive, hence sometimes we shall simply say 
\emph{admissible} in place of reflexive and admissible .

Identity \eqref{2} above  is related to various similar identities.
For example,
Werner \cite{W}  showed, among other, that a variety
$\mathcal V$ 
 is congruence permutable if and only if 
$\mathcal V$ satisfies
$R \circ R \subseteq R$.
This corresponds to the special case
$\Theta=1$ in \eqref{2}.  
Here $1$ or $1_\mathbf A$  denotes the largest congruence
on the algebra $\mathbf A$ under consideration. 
As another example, and an almost immediate consequence of results from
Kazda,  Kozik,  McKenzie and Moore
\cite{kkmm}, we observed in \cite[Proposition 3.1 and p.10]{JDS} 
 that a variety $\mathcal V$  is congruence distributive 
if and only if there is some $k$ such that $\mathcal V$ satisfies 
$\Theta (R \circ S) \subseteq \Theta R  \circ _{k}  \Theta S$.
Taking $R=S$ in this identity, we get \eqref{2} again. 
In this sense, identity \eqref{2} seems still another way
to see that congruence modularity is some kind of a combination
of congruence permutability with distributivity \cite{G}.
 
By the way, notice that congruence distributivity is equivalent also to 
$\alpha ( \beta \circ \gamma ) \subseteq \alpha \beta  \circ _{h} \alpha \gamma  $,
for some $h$ and for congruences $\alpha$, $\beta$,
$\gamma$, by a classical paper by J{\'o}nsson  \cite{JD}.  
On the other hand, both identity \eqref{2} and Werner's identity
$R \circ R = R$ become trivially true in every algebra,
if we let $R$ be a congruence rather than a reflexive and admissible relation.
This shows that many (but not all) relation identities become trivial when 
relations are replaced by congruences. In particular, considering 
relation identities for their own sake seems to give
 a new perspective to the subject.

Motivated by the above considerations and, in particular, 
by the characterization \eqref{2} 
of congruence modularity,
 we looked for
a characterization of congruence distributivity
by means of some expression of the form 
$X (Y \circ Y) \subseteq  something$, where
 in the factor on the left-hand side we are taking
the composition of some $Y$ \textit{with itself}.
At the beginning this looked only like an odd
curiosity; however, after a while, 
a rather clear and, at least in the author's opinion, interesting
picture emerged.

 We found that  congruence distributivity
is equivalent to equation \eqref{2}, for some $h$,  provided that
 $R$ there is taken to be the set-theoretical union of 
some family  of  admissible relations,
rather than an admissible relation.  
Equivalently, we can take $R$ there to be the 
set-theoretical union of two congruences.
Similar ideas are not completely new.
Unions of congruences have been used in 
a proof of J{\'o}nsson's characterization of congruence distributivity,
as presented in McKenzie,  McNulty and  Taylor \cite[Theorem 1.144]{alvin}.
Unions of congruences have been used also in 
  Kaarli and Pixley \cite[Lemma 1.1.12]{KP},
 giving a proof that
every algebra with a compatible near unanimity term
is congruence distributive.
The proof is credited to E. Fried. 

More characterizations of congruence distributivity
are possible in terms of  unions of admissible relations,  
 \emph{U-admissible} relations, for short.
For example, a variety $\mathcal V$ is congruence distributive 
if and only if $\mathcal V$ satisfies 
$(\Theta( \sigma \circ \sigma )) \astt 
 = ( \Theta  \sigma ) \astt $.
Here $ \astt$ denotes transitive closure;
$\sigma$, $\tau$ and $ \upsilon $
shall be used as variables for   
$U$-admissible relations.
Further characterizations shall be presented
in Corollary \ref{cor}.
In most cases, 
$\sigma$, $\tau$ and $ \upsilon $
can be equivalently interpreted
in other ways, for example, 
as unions of two congruences. In each result we shall
explicitly specify the possibilities for 
$\sigma$, $\tau$ and $ \upsilon $.
In any case, both in the present
discussion and in the rest of the paper,
the reader might always assume that $\sigma$, $\tau$ and $ \upsilon $
are $U$-admissible relations.

The question naturally arises
whether in the above formulae the tolerance  $\Theta$,
too, can be equivalently taken to be  a $U$-admissible relation,
still getting a condition equivalent to congruence
distributivity.
This seems to be the main  problem left open by the present note.
On the positive side, we show 
that $\Theta$ can be taken to be a 
$U$-admissible relation in some special cases,
i.e., in varieties with a majority term and
in a $4$-distributive variety introduced by 
Baker \cite{B}. 
This leads to some new characterizations
of arithmetical varieties and of varieties with a majority term.

Dealing with  general results, we show, among other,
that a variety $\mathcal V$ satisfies 
 $\sigma( \tau \circ \upsilon) \subseteq \sigma \tau \circ_h \sigma \upsilon$,
for some $h$, if and only if  
$\mathcal V$ satisfies 
 $\sigma \astt  \tau  \astt  = (\sigma \tau ) \astt  $.
When expressed in terms of tolerances, 
the latter identity is equivalent to congruence 
modularity and has found many applications.
See, e.g.,  Cz\'edli,  Horv\'ath and Lipparini
\cite{CHL}.
We also show that we get equivalent conditions if we consider
the identity  $\sigma( \tau \circ \upsilon) \subseteq \sigma \tau \circ_h \sigma \upsilon$
for $U$-admissible relations
and for admissible relations.
Here the value of $h$ remains the same in both cases.
This is quite surprising; in this paper we shall find many equivalences
of this kind,  starting from Theorem \ref{cdist} below,
but usually the parameters are not constant. 
 
Finally, for any fixed $h$, we shall find a characterization of the 
identity $\Theta( \sigma \circ \sigma ) \subseteq (\Theta \sigma) \alt h $.
 A variety satisfies such an identity for $U$-admissible relations
if and only if there is some function 
$f: \{0, 1,  \dots, h-1 \} \to \{ 1,2\}  $
such that   
$\mathcal V$ satisfies the identity
$\alpha( R_1 \circ R_2 ) \subseteq \alpha  R _{f(0)} \circ 
\alpha  R _{f(1)} \circ \dots \circ  \alpha  R _{f(h-1)}  $
for admissible relations.
This shows that 
the 
identity $\Theta( \sigma \circ \sigma ) \subseteq (\Theta \sigma) \alt h $
is equivalent to a finite union of strong Maltsev classes. It suggests
that the study of relation identities satisfied in congruence distributive varieties
might provide a finer classification of 
such varieties, in comparison with the study of congruence identities
alone.

In conclusion, the equivalences we have found and the 
connections with other kinds of identities suggest 
that the study of $U$-admissible relations 
 has  an intrinsic interest and can be pursued further
in the case of congruence distributive varieties
and probably even in more general contexts.

\section{Congruence distributivity is equivalent
to the identity $ \Theta ( \sigma \circ \sigma ) \subseteq ( \Theta  \sigma ) \alt h$,
for some $h$} \labbel{main}

Recall  that a  binary relation $\sigma$ 
on some algebra $\mathbf A$ is 
\emph{U-admissible} if $\sigma$  
 can be expressed
as the set-theoretical union of some nonempty set of
reflexive and admissible relations on $\mathbf A$ 
(we are not assuming $\sigma$  
itself to be admissible!).
A relation  is 
\emph{U$_2$-admissible} if it  can be expressed
as the  union of two
reflexive and admissible relations.
Let
 $\mathbf F _{ \mathcal V } ( 3) $
  denote the free algebra  in $\mathcal V$ generated by 3 
elements.

\begin{theorem} \labbel{cdist} 
For every variety $\mathcal V$, the following conditions are equivalent.
 Each   condition
holds for $\mathcal V$ if and only if 
it holds for $\mathbf F _{ \mathcal V } ( 3) $.

  \begin{enumerate}[(1)]
    \item 
$\mathcal V$ is congruence distributive;
\item
for some $h$, $\mathcal V$ 
satisfies the identity
$ \Theta ( \sigma \circ \sigma ) \subseteq ( \Theta  \sigma ) \alt h$,
for every tolerance  $\Theta$   and every
U-admissible relation $\sigma$. 
\item
for some $k$, $\mathcal V$ 
satisfies the identity 
$ \alpha  ( \sigma \circ \sigma ) \subseteq ( \alpha   \sigma ) \alt k $,
for every congruence  $\alpha$    and 
every  binary relation $\sigma$ expressible
as the set-theoretical union of two congruences.
  \end{enumerate}
\end{theorem}

  \begin{proof}
For $n=1,2,3$, let
($n$)$_{\mathcal V}$  
denote the respective condition supposed to hold
for $\mathcal V$. Let
($n$)$_{\mathbf F _{ \mathcal V }}$
denote the condition supposed to hold
just for
 $\mathbf F _{ \mathcal V } ( 3) $.

(1)$_{\mathcal V}$  $\Rightarrow $  (2)$_{\mathcal V}$
will follow from Proposition \ref{12} below.

(2) $\Rightarrow $  (3) is trivial for both indices.

(3)$_{\mathbf F _{ \mathcal V }}$ $\Rightarrow $  (1)$_{\mathcal V}$
will follow from Lemma \ref{31} below. 

Since 
($n$)$_{\mathcal V}$  
$\Rightarrow $  ($n$)$_{\mathbf F _{ \mathcal V }}$
trivially,
for $n=1,2,3$, 
and ($1$)$_{\mathbf F _{ \mathcal V }}$
$\Rightarrow $  ($1$)$_{\mathcal V}$  
by J{\'o}nsson \cite{JD},
we get that all the conditions are equivalent. 
 \end{proof}

To prove the implication (3) $\Rightarrow $  (1)
in Theorem \ref{cdist} 
we shall use a classical result by J{\'o}nsson \cite{JD}.
Recall that \emph{J{\'o}nsson terms}  are terms 
$j_0, \dots, j_k$, for $k \geq 1$, satisfying
\begin{align} \labbel{j1} 
\tag{J1}
  x&= j_0(x,y,z),  \quad \ \ \ \ j_{k}(x,y,z)=z   \\
\labbel{j2}
\tag{J2}
  x&=j_i(x,y, x), 
\quad \ \ \  \text{ for } 
0 \leq i \leq k,
 \\ 
 \labbel{j3} \tag{J3} 
\begin{split}  
 j_{i}(x,x,z) &=
j_{i+1}(x,x,z),
 \quad \text{ for even $i$,\ } 
0 \leq i < k,
  \\ 
 j_{i}(x,z,z)&=
j_{i+1}(x,z,z),
  \quad \text{ for odd $i$,\ } 
0 \leq i < k,
 \end{split}
\end{align}

J{\'o}nsson proved that a variety $\mathcal V$ 
is congruence distributive if and only if there is some 
 $k$ such that $\mathcal V$  
 has J{\'o}nsson terms $j_0, j_1, \dots j_k $.    
If this is the case, $\mathcal V$ is said to be
\emph{$k$-distributive} or to be 
$\Delta_k$.  

\begin{lemma} \labbel{31}
If a variety $\mathcal V$ 
(equivalently,  $\mathbf F _{ \mathcal V } ( 3) $)
satisfies condition (3)
in Theorem \ref{cdist} for some specified $k$,
then $\mathcal V$ is $(k+1)$-distributive.   
 \end{lemma}

 \begin{proof}
Let $\alpha$, $\beta$, $\gamma$ be congruences
and $\sigma= \beta \cup \gamma $.
From condition
(3) we get
\begin{multline} \tag{C} \labbel{3}   
\alpha ( \beta \circ \gamma ) \subseteq 
\alpha ( \sigma \circ \sigma )
\subseteq 
( \alpha   \sigma ) \alt k 
=
(\alpha \beta  \cup \alpha \gamma ) \alt k 
\subseteq 
\\
(\alpha \beta  \circ \alpha \gamma ) \circ _ k 
(\alpha \gamma  \circ \alpha \beta  ) =
\alpha \beta  \circ_{k+1} \alpha \gamma,
\end{multline}   
since
$\alpha \gamma $ and $\alpha \beta $ 
are congruences, hence transitive.  
It is a standard and well-known fact 
already implicit in \cite{JD}  that,
within a variety (equivalently, in $\mathbf F _{ \mathcal V } ( 3) $),
 the identity 
$\alpha ( \beta \circ \gamma ) \subseteq 
\alpha \beta  \circ_{k+1} \alpha \gamma$  
is equivalent to
 $(k+1)$-distributivity.
See, e.g.,\ \cite{JDS} for full details. Thus
the lemma is proved. 
 \end{proof}    

In fact, slightly more can be proved.

\begin{lemma}      \labbel{slmore}  
If a variety $\mathcal V$ satisfies condition (3)
in Theorem \ref{cdist} for  some $k$,
then
 $\mathcal V$ satisfies either 
$\alpha ( \beta \circ \gamma ) \subseteq 
\alpha \beta  \circ_{k} \alpha \gamma$  
or 
$\alpha ( \beta \circ \gamma ) \subseteq 
\alpha \gamma   \circ_{k} \alpha \beta $
for congruences.  
\end{lemma} 

\begin{proof}
Let us work in $\mathbf F _{ \mathcal V } ( 3 ) $
with generators $x$, $y$ and  $z$.
Let $ \alpha =Cg(x,z)$, $ \beta =Cg(x,y)$ and  $ \gamma  =Cg(y,z)$,
hence 
$(x,z) \in \alpha ( \beta \circ \gamma )$. 
Since
\begin{equation*}
(\alpha \beta  \cup \alpha \gamma ) \alt k 
\subseteq
(\alpha \beta  \circ_{k} \alpha \gamma) \cup
(\alpha \gamma  \circ_{k} \alpha \beta ), 
 \end{equation*}    
then from the first line in equation \eqref{3}
in the proof of  \ref{31} above, 
we get that either 
$(x,z) \in \alpha \beta  \circ_{k} \alpha \gamma$
or
$(x,z) \in \alpha  \gamma  \circ_{k} \alpha \beta $. 
By the standard homomorphism argument,
we have that, correspondingly,
either
$\alpha ( \beta \circ \gamma ) \subseteq 
\alpha \beta  \circ_{k} \alpha \gamma$  
or 
$\alpha ( \beta \circ \gamma ) \subseteq 
\alpha \gamma   \circ_{k} \alpha \beta $
hold for arbitrary congruences of algebras in $\mathcal V$.
\end{proof}

Compare Lemma \ref{slmore}
with   Theorem  \ref{mal} and 
Remark \ref{a}.

In order to prove (1) $\Rightarrow $  (2)
in Theorem \ref{cdist} we shall
rely heavily on a
recent result by Kazda,  Kozik,  McKenzie and Moore
\cite{kkmm}. 
  \emph{Directed J{\'o}nsson terms},
or \emph{Z\'adori  terms}  
\cite{kkmm,Z} 
are obtained from J{\'o}nsson  conditions
\eqref{j1}-\eqref{j3} above
by replacing 
 condition
\eqref{j3} 
with
\begin{equation}\labbel{D}
\tag{D}
  j_{i}(x,z,z)=
j_{i+1}(x,x,z),  \quad \text{ for \ } 
0 \leq i < k.
  \end{equation}    

To the best of 
our knowledge, directed J{\'o}nsson terms first appeared unnamed
in Z\'adori   \cite[Theorem 4.1]{Z}.
Kazda,  Kozik,  McKenzie and Moore
\cite{kkmm} proved that a variety $\mathcal V$ 
has directed J{\'o}nsson terms for some $k$
if and only if $\mathcal V$  has 
J{\'o}nsson terms
for some $k'$. In particular, they obtained 
 that  a variety $\mathcal V$ is congruence distributive if and only if
$\mathcal V$  has directed J{\'o}nsson terms for some $k$.
Notice that here we are using a slightly 
different indexing of the terms,
in comparison with \cite{kkmm}.
The argument below showing that we can switch from 
congruences to tolerances is due to 
Cz\'edli and  Horv\'ath \cite{CH}.

\begin{proposition} \labbel{12}
If $\mathcal V$ is congruence distributive with
directed J{\'o}nsson terms $d_0, d_1, \dots, d_n $,
then $\mathcal V$ satisfies condition (2) in Theorem \ref{cdist}
with $h=2n-2$. 
 \end{proposition} 

 \begin{proof}
Fix some algebra $\mathbf A$ in $\mathcal V$ and suppose that
 $\Theta$ is a tolerance on $\mathbf A$ 
and $\sigma$ is a union of reflexive and admissible relations
on $\mathbf A$.
If $(a,c) \in  \Theta ( \sigma \circ \sigma )$,
then $ a \mathrel \Theta c  $ and 
there is $b$ such that 
$ a \mathrel \sigma b \mathrel \sigma c$.   
By using directed J{\'o}nsson terms, we get
 \begin{multline*}
 a = d_0(a,c,c) = d_1(a,a,c) \mathrel { \sigma  } d_1(a,b,c)
 \mathrel { \sigma  } d_1(a,c,c)=
\\
d_2(a,a,c)  \mathrel { \sigma  } d_2(a,b,c)
 \mathrel { \sigma  } \dots  \mathrel { \sigma  }d_{n-1}(a,b,c) \mathrel \sigma
d_{n-1}(a,c,c) = d_n(a,a,c) =c.
\end{multline*}   
Indeed, by assumption,
 $\sigma = \bigcup _{g \in G} \sigma _g $,
for some family $\{ \sigma _g \mid g \in G\}$ 
of reflexive and admissible relations.
Since $ a \mathrel \sigma b$,
then    $ a \mathrel \sigma_g b$, for some $g \in G$,
hence  $d_i(a,a,c)  \mathrel { \sigma _g } d_i(a,b,c)$, 
for every $i$, since 
$\sigma_g$ is reflexive and admissible. Thus we get
 $d_i(a,a,c)  \mathrel { \sigma } d_i(a,b,c)$
 from  $\sigma = \bigcup _{g \in G} \sigma _g $.
Similarly,
$d_i(a,b,c)  \mathrel { \sigma  } d_i(a,c,c)$.
The point is that when we go, say, from
 $d_i(a,a,c) $ to $  d_i(a,b,c)$,
only \emph{one}   
element is moved. In this case, the 
$a$ in the second position is changed to $b$ and all the other arguments
are left unchanged.  
Moreover,
\begin{equation*}
d_i(a,a,c) = d_i(d_i(a,b,a),a,d_i(c,b,c) )
\mathrel \Theta  d_i(d_i(a,b,c),a,d_i(a,b,c) ) =d_i(a,b,c) 
\end{equation*}    
 and similarly 
$d_i(a,b,c) 
\mathrel \Theta  d_i(a,c,c) $.
This is essentially an argument 
from \cite{CH}.
\end{proof} 

By iterating 
the formula 
$ \Theta ( \sigma \circ \sigma ) \subseteq ( \Theta  \sigma ) \alt h$,
we get, say, 
\begin{equation*}
 \Theta ( \sigma \circ \sigma \circ \sigma \circ \sigma )
 \subseteq ( \Theta  (\sigma \circ \sigma  )) \alt h \subseteq 
( \Theta  \sigma )^{h^2},
  \end{equation*}    
where we have used the fact that if $\sigma$ and $\tau$ 
are U-admissible, then $\sigma \circ \tau $ 
is  U-admissible, in particular, 
$\sigma \circ \sigma  $ 
is  U-admissible. 
Indeed, if $\sigma = \bigcup _{g \in G} \sigma _g $
and 
$ \tau  = \bigcup _{f \in F} \tau  _f $,
then
$\sigma \circ \tau 
 = \bigcup _{g \in G, f \in F} (\sigma _g \circ \tau _f)$. 
However, we get better bounds by adapting the proof of 
Proposition \ref{12}, as we are going to show.

\begin{proposition} \labbel{12b}
Suppose that $\mathcal V$ is a congruence distributive 
variety with
directed J{\'o}nsson terms $d_0, d_1, \dots, d_n $.
  \begin{enumerate}[(1)]
    \item  
For every $m$, 
$\mathcal V$ satisfies the identity
$ \Theta \sigma  \alt m  \subseteq ( \Theta  \sigma ) \alt {(mn-m)}$.
\item
More generally, for $m$ even, 
$\mathcal V$ satisfies 
$ \Theta( \sigma \circ_m \tau )  \subseteq \Theta  \sigma \circ_{mn-m}
\Theta \tau $.
  \end{enumerate}
In the above identities,
$\Theta$ is a tolerance and 
$\sigma$ is a U-admissible relation.
 \end{proposition}

  \begin{proof} 
 The proof is similar to the proof of Proposition \ref{12}. 
We shall prove (2). When $m$ is even, (1) is the special case
$\sigma= \tau $ 
 of
(2). When $m$ is odd, (1) allows a similar proof.  

Suppose that $(a,c) \in  \Theta( \sigma \circ_m \tau ) $.
Then $ a \mathrel \Theta c  $ and 
$ a=b_0 \mathrel \sigma b_1 \mathrel \tau  b_2 \mathrel { \sigma }
b_3 \mathrel { \tau } 
\dots \mathrel { \sigma }  b _{m-1} \mathrel { \tau  }  b _{m} = c $,
for certain $b_0, b_1, \dots$\    
We have 
$b _{m-1} \mathrel { \tau  }  b _{m}$ 
since $m$ is even. 
For $\ell=1, \dots, n-1$, arguing as in  
the proof of Proposition \ref{12}, we get
 \begin{multline*}\labbel{} 
  d_\ell(a,a,c) = d_\ell(a,b_0,c) \mathrel { \sigma  } d_\ell(a,b_1,c) \mathrel \tau   
 d_\ell(a,b_2,c) \mathrel { \sigma } \dots \\
 \mathrel \tau   d_\ell(a,b_m,c)= 
d_\ell(a,c,c)
 =d_{\ell+1}(a,a,c).
 \end{multline*}    
All the elements in the 
above chain are $\Theta$-related,
since 
\begin{multline*}
d_\ell(a,b_h,c) = d_\ell(d_\ell(a,b_k,a),b_h,d_\ell(c,b_k,c) )
\mathrel \Theta \\
 d_\ell(d_\ell(a,b_k,c),b_h,d_\ell(a,b_k,c) )  =d_\ell(a,b_k,c),
\end{multline*}    
for all indices $h $, $  k$.  
This shows that 
 $( d_\ell(a,a,c), d_{\ell+1}(a,a,c)) \in \Theta  \sigma \circ _m  \Theta \tau $,
for $\ell=1, \dots, n-1$.
Since 
 $a = d_0(a,c,c ) =  d_1(a,a,c)$
and $d_n(a,a,c) =c $,
then, by putting everything together, 
 we get a chain of length 
$m(n-1)$ 
from $a$ to $c$.
Notice that the factors 
of the form 
$\Theta  \sigma $ and $  \Theta \tau$ 
 do always alternate, since $m$ is assumed to be even.
Indeed,   
 $d_\ell(a,b_{m-1},c) \mathrel \tau  
d_\ell(a,c,c)
 =d_{\ell+1}(a,a,c)
\mathrel \sigma d_{\ell+1}(a,b_1,c)$.
\end{proof}

If we restrict ourselves to relations $\sigma$ 
which are unions of tolerances, 
the above arguments can be carried over just using
J{\'o}nsson terms.
Let 
$S \mathrel {_m\circ} T$  denote
$  \dots T \circ S \circ T $ with $m$ factors. 
Formally,
$S \mathrel {_m\circ} T= S \circ _m T$
if $m $ is even and  
$S \mathrel {_m\circ} T= T \circ _m S$
if $m$ is odd. 
Let $^\smallsmile $ denotes  \emph{converse}.

\begin{proposition} \labbel{12c}
If $\mathcal V$ is a $k$-distributive variety,
then the following statements hold.

  \begin{enumerate}[(1)]   
 \item 
For every $m$, 
$\mathcal V$ satisfies the identity
$ \Theta \sigma  \alt m  \subseteq ( \Theta  \sigma ) \alt {(mk-m)}$,
for every  tolerance $\Theta$ and 
every 
 relation $\sigma$ which is a union of tolerances.

\item
For every $m$,
$\mathcal V$ satisfies 
\begin{equation*}\labbel{kd}
      \Theta( \sigma \circ_m \tau )  \subseteq 
(\Theta  \sigma \circ_m  \Theta \tau ) \circ_{k-1}
(\Theta \tau  ^\smallsmile \mathrel {_m\circ} \Theta \sigma ^\smallsmile ),
  \end{equation*}
for every  tolerance $\Theta$ and 
all
  U-admissible relations $\sigma$ and $\tau$. 
  \end{enumerate}    

In turn, if a variety satisfies (2) above in the particular case
$m=2$, then 
$\mathcal V$ is $k$-distributive.
 \end{proposition}

\begin{proof}
(1) is the particular case of (2)
when $ \sigma = \tau $, since 
if $\sigma$ is a union of tolerances, then 
$\sigma = \sigma ^\smallsmile $.
Hence it is enough to prove (2).
 
Let $j_0, \dots,  j_k$ be J{\'o}nsson terms.
Under the same assumptions as in the proof of 
\ref{12b}, we get
 $( j_\ell(a,a,c), j_{\ell}(a,c,c)) \in \Theta  \sigma \circ _m  \Theta \tau $,
for $\ell=1, \dots, k-1$ and, taking converses,
 $( j_\ell(a,c,c), j_{\ell}(a,a,c)) \in
 \Theta \tau  ^\smallsmile \mathrel {_m\circ} \Theta \sigma ^\smallsmile$.
 In the present situation we have 
$ j_\ell(a,a,c)= j_{\ell+1}(a,a,c)$, for $\ell$ even, 
and  $ j_\ell(a,c,c)= j_{\ell+1}(a,c,c)$, for $\ell$ odd.
Hence, in order to join the partial chains, we have to consider
$\Theta  \sigma \circ _m  \Theta \tau $
and 
$\Theta \tau  ^\smallsmile \mathrel {_m\circ} \Theta \sigma ^\smallsmile$,
alternatively.

In order to prove the last sentence, take 
$  \alpha =\Theta$, $\beta= \sigma $
and $ \gamma = \tau $ congruences.
In the displayed formula in (2) we then get adjacent
occurrences of      
$\alpha \gamma $ and 
$\alpha \beta $
which pairwise absorb.
We end up
with  
$\alpha ( \beta \circ \gamma ) \subseteq
\alpha \beta  \circ _{k} \alpha \gamma $,  
an identity equivalent to 
$k$-distributivity. 
\end{proof}

Recall that $ \astt $ denote transitive closure.

\begin{corollary} \labbel{cor} 
Within a variety $\mathcal V$, each of the following identities
is equivalent to congruence distributivity.
 Each   identity
holds in $\mathcal V$ if and only if 
it holds in $\mathbf F _{ \mathcal V } ( 3) $.
  \begin{enumerate}   
\item[(1)]
$\Theta( \sigma \circ \sigma )
 \subseteq  ( \Theta  \sigma  ) \astt $, equivalently, 
$(\Theta( \sigma \circ \sigma )) \astt 
 = ( \Theta  \sigma ) \astt $;
\item[(1$'$)]
$\Theta( \sigma \circ \sigma )
 \subseteq  ( \Theta  \sigma  \circ \Theta  \sigma ^\smallsmile  ) \astt $;
\item[(1$''$)]
$\Theta( \sigma \circ \sigma ^\smallsmile )
 \subseteq  ( \Theta  \sigma  \circ \Theta  \sigma ^\smallsmile  ) \astt $,
equivalently,
$(\Theta( \sigma \circ \sigma ^\smallsmile )) \astt 
= ( \Theta  \sigma  \circ \Theta  \sigma ^\smallsmile  ) \astt $;
 \item[(2)]   
$ \Theta \sigma  \astt   \subseteq ( \Theta  \sigma ) \astt $, equivalently,
$ (\Theta \sigma  \astt ) \astt   = ( \Theta  \sigma ) \astt $;
\item[(3)]
$\Theta( \sigma \circ \tau )
 \subseteq  ( \Theta  \sigma  \circ
\Theta \tau ) \astt $, equivalently, 
$(\Theta( \sigma \circ \tau )) \astt 
 = ( \Theta  \sigma  \circ
\Theta \tau ) \astt $;
\item[(4)]
$ \Theta( \sigma \circ \tau )  \astt   \subseteq ( \Theta  \sigma  \circ
\Theta \tau ) \astt $, equivalently, 
$ (\Theta( \sigma \circ \tau )  \astt ) \astt  
 = ( \Theta  \sigma  \circ
\Theta \tau ) \astt $;
\item[(4$'$)]
$ \Theta( \sigma \circ \tau )  \astt   \subseteq ( \Theta  \sigma  \circ
\Theta \tau \circ \Theta  \sigma ^\smallsmile   \circ
\Theta \tau ^\smallsmile  ) \astt $, equivalently, 

\noindent
$ (\Theta( \sigma \circ \tau  \circ \sigma ^\smallsmile \circ \tau ^\smallsmile )  \astt ) \astt  
 = ( \Theta  \sigma  \circ
\Theta \tau \circ \Theta  \sigma ^\smallsmile   \circ
\Theta \tau ^\smallsmile  ) \astt $.
 \end{enumerate}
In the above identities,
$\Theta$ is a tolerance, equivalently, a congruence, and 
$\sigma$ and $\tau$ are U-admissible relations, 
equivalently,
relations  which are 
expressible as the union of two congruences.
\end{corollary}

 \begin{proof}
For each of the above conditions, 
let the \emph{strong form} be the one in which 
 $\Theta$  is supposed to be a tolerance and
$\sigma$, $\tau$  
are U-admissible relations.
Let the \emph{weak form} be the one in which
 $\Theta$  is supposed to be a congruence and  
$\sigma$ and $\tau$ are
 relations which are the union of two congruences.
Clearly, in each case, the strong form implies the weak form.

Inside each item, in each case the  conditions
are equivalent, since $\sigma$ and $\tau$  are reflexive
and since transitive closure is a monotone
and idempotent operator.
A further comment is necessary for (4$'$).
To get the last condition, first apply the first condition
with $\sigma \circ \tau $ in place of $\sigma$ 
and  with $\sigma ^\smallsmile \circ \tau ^\smallsmile  $ in place of
$\tau$, then apply again the first condition.

Congruence distributivity implies the strong form of
(4), as a consequence of  
 Proposition \ref{12b}(2).
Trivially, (4) $\Rightarrow $  (4$'$), 
(4) $\Rightarrow $  (1$''$),
(4) $\Rightarrow $  (3) $\Rightarrow $  (1) $\Rightarrow $  (1$'$)
 and (4) $\Rightarrow $  (2) 
$\Rightarrow $  (1),
either in the strong or in the weak case.
Moreover, the weak forms of (4) and (4$'$)
are equivalent, since if $\sigma$ is a union of congruences, then 
$\sigma= \sigma ^\smallsmile $ and the same for $\tau$.
Similarly, the weak forms of (1), (1$'$) and (1$''$)
are equivalent.

Hence, in order to close the cycles of implications,
it is enough to prove that
 the weak form of (1) implies congruence distributivity.
As in the proof of Lemma \ref{31}, take $ \Theta = \alpha $
and  $\sigma= \beta \cup \gamma $, with $\alpha, \beta $ and 
$ \gamma $ congruences. Then the weak form of (1) implies 
\begin{equation*}\labbel{} 
\alpha( \beta \circ \gamma ) \subseteq 
\alpha ( \sigma \circ \sigma ) \subseteq 
(\alpha \sigma ) \astt  =
(\alpha \beta \cup \alpha \gamma) \astt 
=\alpha \beta + \alpha \gamma. 
 \end{equation*} 
 
This condition 
(even only  in $\mathbf F _{ \mathcal V } ( 3 ) $)
implies congruence distributivity 
by J{\'o}nsson \cite{JD}. 
 \end{proof}  

By \cite{kkmm},
Proposition \ref{12b} and
 the above proof  we have that a variety
$\mathcal V$ 
 is congruence distributive
if and only if, for some (equivalently every)
$m \geq 2$, $\mathcal V$ satisfies
the identity 
$\Theta( \sigma \circ_m \sigma )
 \subseteq  ( \Theta  \sigma  ) \astt $.
A similar remark applies to all the conditions in
Corollary \ref{cor}, except for (2).

\section{Working with U-admissible relations only and a major problem}
 \labbel{prob}

Notice that we do need to consider the possibility for  $\sigma$ 
to be a U-admissible (not just admissible)
 relation in Theorem \ref{cdist}
and  in Corollary \ref{cor} (1)-(2), in order to obtain  conditions 
equivalent to congruence distributivity.
Of course, if an identity holds for U-admissible relations,
then it holds for admissible relations, too,
hence both versions of the identities
we have introduced hold in congruence distributive varieties.
However, when expressed in terms of
 admissible relations only, the identities might turn out 
to be too weak
 to imply back  congruence distributivity.
Indeed, as we mentioned in the introduction, $R \circ R = R$ holds in congruence permutable varieties,
for every reflexive and admissible relation $R$. A fortiori, 
$\Theta (R \circ R) = \Theta R \subseteq ( \Theta R) \astt $  hold,
but there are  congruence permutable varieties which are not
 congruence distributive.
Hence the identities in 
\ref{cdist}(2)(3) and 
 \ref{cor} (1)(1$'$)(2)
do not imply congruence distributivity, if 
$\sigma$ is taken to be an admissible relation.
In congruence permutable varieties 
$ R ^\smallsmile = R$ 
holds, too, again by Werner \cite{W}.
This shows that the identity in  \ref{cor}(1$''$) considered only
for admissible relations fails to imply congruence distributivity, as well.

On the other hand, we do not know 
the exact possibilities for  $\Theta$ in the identities we have discussed. 

\begin{problem} \labbel{mainp} 
(a) Do we get conditions equivalent to 
congruence distributivity if we take a reflexive and admissible relation 
$T$ 
(or even a U-admissible relation $\upsilon$)
in place of $\Theta$ in Theorem \ref{cdist} and Corollary \ref{cor}?

(b) Similarly, do we get a condition equivalent to congruence modularity
if we take a reflexive and admissible relation 
$T$ in place of $\Theta$ in equation \eqref{2}
from the introduction?  
 \end{problem}

The results of the present section
show that in many cases the answer
to Question (a) above is the same when asked for 
  admissible relations 
or for  U-admissible relations.
In any case, the next three  propositions 
provide examples of varieties in which
 corresponding identities hold true.
In passing, we also get  some further characterizations 
of varieties 
with a majority term and of
arithmetical varieties.

\begin{proposition} \labbel{maj}
For every variety $\mathcal V$, the following conditions are equivalent. \begin{enumerate}[(1)]
   \item 
$\mathcal V$ has a majority term.
\item
$\mathcal V$ satisfies 
$\alpha ( \beta \circ \gamma ) \subseteq \alpha \beta \circ \alpha \gamma $
for congruences. 
\item
$\mathcal V$ satisfies 
$\sigma ( \tau \circ \upsilon) \subseteq  \sigma  \tau \circ \sigma  \upsilon $
 for U-admissible relations, equivalently, for 
reflexive and admissible relations, equivalently, for tolerances.
  \end{enumerate}
\end{proposition} 

 \begin{proof} 
The equivalence of (1) and (2) is standard and well-known.
If $\mathcal V$  satisfies (3)
for $U$-admissible relations, then 
$\mathcal V$ obviously satisfies (3) for admissible relations and 
for tolerances. In all cases this implies (2).
In order to prove that (1) implies the strongest form of (3),
let $m$ be a majority term. If 
$(a,c) \in \sigma ( \tau \circ \upsilon)$
with
$ a \mathrel \sigma c$
and  
$a \mathrel \tau b \mathrel {\upsilon} c $, then
$a= m(a,a, c) \mathrel \tau m(a,b,c)$ 
and
$a= m(a,b, a) \mathrel \sigma  m(a,b,c)$.
Hence $a \mathrel {\sigma \tau } m(a,b,c)$.
Here, as in the proof of Proposition \ref{12},
we are using the fact that only one element is moved at a time.   
Symmetrically,
$ m(a,b,c) \mathrel { \sigma \upsilon } c $.
Hence the element 
$m(a,b,c)$ witnesses 
$(a,c) \in \sigma \tau \circ \sigma  \upsilon$.
\end{proof} 

We are now going to prove the result analogous to 
\ref{maj} for arithmetical varieties, but first a
remark is in order.

\begin{remark} \labbel{majvsarith} 
Clearly, we can have equality in place of inclusion in 
\ref{maj}(2); actually,
\textit{as far as $\sigma$ is transitive}, we always have
 $\sigma ( \tau \circ \upsilon) \supseteq  \sigma  \tau \circ \sigma  \upsilon $.
On the other hand, for  U-admissible relations, 
the identity 
$\sigma ( \tau \circ \upsilon) =  \sigma  \tau \circ \sigma  \upsilon $
holds only in the trivial variety,
as we are going to show.
First observe that, by taking $\tau= \upsilon=1$,
we get $\sigma= \sigma \circ \sigma $.  

We now show that, for  U-admissible relations,
$\sigma= \sigma \circ \sigma $ holds only in the trivial variety.
Let $\mathbf A$ have at least two elements, 
consider the algebra $\mathbf B =\mathbf A \times \mathbf A$ 
and let $\sigma$ be the union of the kernels of the two projections.
Then $\sigma \neq 1 _{\mathbf B} = \sigma \circ \sigma  $. 
Actually, we have shown that
$\sigma= \sigma \circ \sigma $ fails in any non-trivial  variety
already for $\sigma$ a union of two congruences.

Turning to \textit{admissible} relations,
we are going to show that we can actually have
$T(R \circ S) = TR \circ TS$.
Curiously, this identity is stronger than 
having a majority term; indeed it is equivalent to arithmeticity. 
On the other hand, the inclusion
$T(R \circ S) \subseteq  TR \circ TS$
is  equivalent 
to having a majority term,
by \ref{maj}(3). 
\end{remark}      

\begin{proposition} \labbel{arith}
For every variety $\mathcal V$, the following conditions are equivalent. \begin{enumerate}[(1)]
   \item 
$\mathcal V$ is arithmetical, that is, $\mathcal V$ has a Pixley term.
\item
$\mathcal V$ satisfies 
$\alpha ( \beta \circ \gamma ) \subseteq  \alpha \gamma  \circ \alpha \beta $
for congruences.
\item
$\mathcal V$ satisfies 
$\sigma ( \tau \circ \upsilon) \subseteq  \sigma  \upsilon\circ \sigma  \tau  $
 for U-admissible relations, equivalently, for 
reflexive and admissible relations, equivalently, for tolerances.
\item
$\mathcal V$ satisfies
$T(R \circ S) = TR \circ TS$,
for reflexive and admissible relations, 
equivalently, for tolerances
(notice that in the present identity we use equality and  $R$ 
and $S$ are not shifted).
  \end{enumerate}
\end{proposition} 

\begin{proof}
Again, (1) $\Leftrightarrow $ (2) is well-known and standard.
In each case, (3) $\Rightarrow $  (2).

In order to show that (1) implies (3) an additional argument is needed,
in comparison with \ref{maj}.
Since (1) implies congruence permutability,
we have  $R = R ^\smallsmile $, for
admissible relations, by \cite{W}. 
This implies 
$ \sigma  = \sigma  ^\smallsmile $, for
$U$-admissible relations.
Then we can proceed as usual employing the Pixley term $p$.
If $ a \mathrel \sigma c$
and  
$a \mathrel \tau b \mathrel {\upsilon} c $, then, for example
$a= p(a,c,c) \mathrel {\upsilon ^\smallsmile } p(a,b,c)$.
Since we have just proved that  $\upsilon ^\smallsmile = \upsilon$,
then  $a \mathrel \upsilon   p(a,b,c)$. With similar computations we show
that   $a \mathrel {\sigma \upsilon  } p(a,b,c) \mathrel {\sigma \tau   } c$.

Finally, we show that arithmeticity is equivalent to (4).
We shall use the  fact that arithmeticity is equivalent to the existence of a majority term
together with  congruence permutabilty.
In congruence permutable varieties any reflexive and admissible relation 
is a congruence, hence (4) follows from arithmeticity by
\ref{maj}(2) and the first sentence in Remark \ref{majvsarith}.
Conversely, if (4) holds, 
then taking $R=S=1$ we get
$T= T \circ T$ and this identity implies congruence permutability 
by \cite{W}. Trivially (4) implies    
\ref{maj}(2), hence we get a majority term.
\end{proof}
 
\begin{proposition} \labbel{baker} 
The identity
$\sigma ( \tau \circ \upsilon) \subseteq \sigma  \tau \circ \sigma  \upsilon
 \circ \sigma  \tau \circ \sigma  \upsilon$ holds
for U-admissible relations
 in 
 the variety introduced by 
Baker \cite{B}, the variety generated by 
polynomial  reducts of lattices in which 
the ternary operation $ f(a,b,c) =a \wedge (b \vee c)$ 
is the only fundamental operation.
 \end{proposition} 

 \begin{proof}
If, as usual,  
$(a,c) \in \sigma ( \tau \circ \upsilon)$
with
$a \mathrel \tau b \mathrel {\upsilon} c $, then
 the elements 
$f(a,b,c)$, $a \wedge c = f(a,c,c)$ and $f(c,b,a)$
witness    
$(a,c) \in  \sigma  \tau \circ \sigma  \upsilon
 \circ \sigma  \tau \circ \sigma  \upsilon $.
 For example, 
\begin{equation*}
f(a,b,c) = f(f(a,b,c),f(a,b,c),a)
\mathrel \sigma 
f(f(c,b,c),f(a,b,c),a) =  a \wedge c. \qedhere
 \end{equation*}    
 \end{proof}

Of course, Proposition \ref{baker}
applies also to any expansion of Baker's variety. 

We now show that equivalences similar to those in 
Propositions \ref{maj} and \ref{arith}  
hold in a more general context.
In particular, we shall prove
the quite surprising fact  that
the identity
$\sigma( \tau \circ \upsilon) \subseteq \sigma \tau \circ_h \sigma \upsilon  $,
holds for U-admissible relations
if and only if it holds for admissible relations,
where $h$ is the same in both cases.

\begin{theorem} \labbel{gen}
Within a  variety, the following identities are equivalent.
Each   identity
holds in $\mathcal V$ if and only if 
it holds in $\mathbf F _{ \mathcal V } ( 3) $.

 \begin{enumerate}[(1)]   
\item
$\sigma( \tau \circ \upsilon) \subseteq (\sigma \tau \circ \sigma \upsilon) \astt  $,
equivalently, 
$(\sigma( \tau \circ \upsilon)) \astt  = (\sigma \tau \circ \sigma \upsilon) \astt  $;
\item
$\sigma( \tau \circ \tau ) \subseteq (\sigma \tau ) \astt  $,
equivalently, $(\sigma( \tau \circ \tau ) ) \astt =(\sigma \tau ) \astt  $;
\item
 $\sigma \astt  \tau  \astt  = (\sigma \tau ) \astt  $,
equivalently,  $(\sigma \astt  \tau  \astt ) \astt  = (\sigma \tau ) \astt  $.
   \end{enumerate}

In the above identities,
 $\sigma$, $\tau$ and $\upsilon$
are U-admissible relations, equivalently,
 U$_2$-admissible relations.
Limited to item (1),
we can equivalently take them to be
reflexive and admissible relations.  
 \end{theorem}  

In order to prove Theorem  \ref{gen}
we first need a Maltsev characterization of condition (1). 
In the formulae below
 we sometimes use a semicolon in place of a comma
 in order
to improve readability.

\begin{proposition} \labbel{vr}
Let $\sigma$, $\tau$, $\upsilon$
denote U-admissible relations,
equivalently, reflexive and admissible relations.
 A variety $\mathcal V$ satisfies
 $\sigma( \tau \circ \upsilon) \subseteq \sigma \tau \circ_h \sigma \upsilon$
if and only if   $\mathcal V$ 
has 
ternary terms $t_i$, $i=0, \dots, h$, 
and $4$-ary terms $s_i$ and $u_i$,
$i=0, \dots, h-1$,
such that the following identities hold throughout $\mathcal V$.
\begin{align*}
x&=t_0(x,y,z), &
t_h(x,y,z) &=z,
\\
t_{i}(x,y,z) &= u_{i}(x,y,z; x), &  
u_{i}(x,y,z; z) &=  t_{i+1}(x,y,z), &&\text{for } i<h,
\\
t_{i}(x,y,z) &= s_{i}(x,y,z; x), &  
s_{i}(x,y,z; y) &=  t_{i+1}(x,y,z), &&\text{for even  } i<h,
\\
t_{i}(x,y,z) &= s_{i}(x,y,z; y), &
s_{i}(x,y,z; z) &=  t_{i+1}(x,y,z), &&\text{for odd } i<h.
 \end{align*}
All the above conditions hold in $\mathcal V$ if and only if 
they hold in $\mathbf F _{ \mathcal V } ( 3) $.
  \end{proposition}  

\begin{proof}
Suppose that 
 $\sigma( \tau \circ \upsilon) \subseteq (\sigma \tau \circ_h \sigma \upsilon)$
holds in $\mathcal V$, or just in $\mathbf F _{ \mathcal V } ( 3) $,
 in the weak form, that is
for reflexive and admissible relations.
Let  $\mathbf F _{ \mathcal V } ( 3) $
  be generated by the 
elements $x$, $y$ and $z$ and
let $\sigma$, $\tau$, $\upsilon$ be
the smallest reflexive and admissible relations containing,
respectively,      
$(x,z)$, $(x,y)$,  $(y,z)$, thus
$(x,z) \in \sigma (\tau \circ \upsilon)$.
By the assumption,
$(x,z) \in \sigma \tau \circ_h \sigma \upsilon$.
Since we are in the free algebra generated by 
 $x$, $y$, $z$, the above relation is witnessed by 
ternary terms
$t_i$, $i=0, \dots, h$.
The elements 
$t_i(x,y, z)$ of $\mathbf F _{ \mathcal V } (3) $
have to satisfy
$x=t_0(x,y,z) $,  
$t_{i}(x,y,z) \mathrel { \sigma \tau }  t_{i+1}(x,y,z)$, for $i$ even,
$t_{i}(x,y,z) \mathrel { \sigma \upsilon }  t_{i+1}(x,y,z)$, for $i$ odd,
 and
$ t_h(x,y,z) =z$. 
It is easy to see that in $\mathbf F _{ \mathcal V } (3) $
we have  
$ \sigma =
 \{ (u(x,y,z;x),u(x,y,z;z)) \mid u \text{ a $4$-ary term of } \mathcal V \}$. 
Hence the relation 
$t_{i}(x,y,z) \mathrel { \sigma }  t_{i+1}(x,y,z)$
implies the existence of terms $u_i$ satisfying the desired equations.
The equations hold throughout $\mathcal V$ since we are in the free 
algebra generated by $3$ elements and the terms depend on $3$ variables.
Arguing in the same way with 
$\tau$ and $\upsilon$,
we get all the claimed equations. 

Now suppose that the displayed equations hold in $\mathcal V$.
The arguments familiar by now
show that 
 $\sigma( \tau \circ \upsilon) \subseteq (\sigma \tau \circ_h \sigma \upsilon)$
holds in the strong form in which
$\sigma$, $\tau$ and $\upsilon$
are taken to be  U-admissible relations.
Indeed, let 
$\mathbf A \in \mathcal V$,
 $\sigma$, $\tau$ and $\upsilon$
be U-admissible relations of $\mathbf A$
and $(a,c) \in \sigma (\tau \circ \upsilon)$,
thus 
$a \mathrel \tau b \mathrel  \upsilon c$, for some $b$. 
Then we have,
say,
$t_{i}(a,b,c) = u_{i}(a,b,c; a) \mathrel  \sigma   
u_{i}(a,b,c; c)  =  t_{i+1}(a,b,c)$.
By using the other equations in a similar way we get that 
the elements 
 $t_{i}(a,b,c)$, $i=0, \dots , h$
witness  
$(a,c) \in \sigma \tau \circ_h \sigma \upsilon$.

Since the stronger form of 
 $\sigma( \tau \circ \upsilon) \subseteq (\sigma \tau \circ_h \sigma \upsilon)$
 implies the weaker form, we have that 
all the conditions are equivalent.
 \end{proof} 

Of course, from the equations  in Proposition \ref{vr}
the $t_i$'s are determined either by the $u_i$'s or by
the $s_i$'s, hence, in principle, it is not even necessary
to mention the $t_i$'s explicitly, it would be enough to reformulate
the equations as follows.
\begin{gather*} \labbel{gath}    
x=u_0(x,y, z; x) = s_0(x,y, z; x), \quad \quad  u_{h-1}(x,y, z; z)=z, \\   
s_{h-1}(x,y, z; z)=z \text{ if $h$ is even, } \quad
s_{h-1}(x,y, z; y)=z \text{ if $h$ is odd,}  \\ 
 u_i(x,y, z; z) = u_{i+1}(x,y, z; x)
=s_i(x,y, z; y) = s_{i+1}(x,y, z; y) \text{ if  $i<h-1$, $i$  even,} \\
 u_i(x,y, z; z) = u_{i+1}(x,y, z; x)
=s_i(x,y, z; z) = s_{i+1}(x,y, z; x) \text{ if $i<h-1$, $i$ odd}.
 \end{gather*} 

 However, both the statement and the proof 
seem  clearer if we mention explicitly the $t_i$'s.

  \begin{proof}[Proof of Theorem  \ref{gen}] 
Clearly, (1) $\Rightarrow $  (2) and
(3) $\Rightarrow $  (2) either in the weak or in the 
strong case.
If (2) holds for U-admissible relations, then, by an easy induction, we get
 \begin{equation}\labbel{4} \tag{P}      
\sigma \tau  \astt  \subseteq  (\sigma \tau ) \astt. 
\end{equation}
Applying (2) again with $\tau \astt $ in place of 
$\sigma$, we get
\begin{equation*}     
\tau  \astt  ( \sigma \circ \sigma ) \subseteq  ( \tau \astt  \sigma  ) \astt =
( \sigma \tau \astt  ) \astt   \subseteq (\sigma \tau ) {\astt} {\astt}
 = (\sigma \tau ) \astt ,
 \end{equation*}
by  \eqref{4} and idempotence  of $ \astt $.
Then another induction shows 
$ \sigma  \astt \tau  \astt = \tau  \astt  \sigma  \astt  \subseteq (\sigma \tau ) \astt $.
The reverse inclusion is trivial.  Thus we have that (2) and (3) are equivalent
for U-admissible relations.

Since, by Proposition \ref{vr},
the weaker and the stronger forms of (1)
are equivalent, it is enough to prove that the 
version of (2) for U$_2$-admissible relations
implies the version of (1)  for reflexive and admissible relations
(for this in turn implies the stronger form of (1), hence the stronger
form of (2), hence (3) follows also from the version of (2) 
only for U$_2$-admissible relations).  

The proof that (2) for U$_2$-admissible relations
implies  (1)  for reflexive and admissible
 relations is similar to the proof of Lemma \ref{31}.
Indeed, for $R$, $S$, and $T$ reflexive and admissible relations,
 letting $\sigma= S \cup T$ and applying (2),
we have
$R(S \circ T) \subseteq R( \sigma \circ \sigma )
\subseteq (R  \sigma ) \astt  = (RS \cup RT) \astt 
= (RS \circ RT) \astt $.
\end{proof}

\section{A Maltsev-like characterization 
of $\alpha( \sigma \circ \sigma ) \subseteq (\alpha \sigma ) \alt h$} \labbel{malsec}

In this section we present a Maltsev type characterization of 
the identity 
$\alpha( \sigma \circ \sigma ) \subseteq (\alpha \sigma ) \alt h$,
for fixed $h$.
This is somewhat similar to Theorem  \ref{gen} and Proposition  \ref{vr},
but qualitatively different in that, for fixed $h$,  the identity 
turns out to be equivalent to a finite disjunction of  identities
involving reflexive and admissible relations.
In other words, for the reader who knows the terminology,
while the identities in Theorem  \ref{gen}
are described by Proposition \ref{vr}
as Maltsev classes, on the other hand, for fixed $h$,   
we describe 
$\alpha( \sigma \circ \sigma ) \subseteq (\alpha \sigma ) \alt h$ as a 
finite union
of strong Maltsev classes.

\begin{theorem} \labbel{mal}
For every natural number $h>0$
and every  variety $\mathcal V$, the following conditions are equivalent. 
Each condition holds for $\mathcal V$ if and only if 
it holds for $\mathbf F _{ \mathcal V } ( 3 ) $.

(1)
$\mathcal V$ 
 satisfies the identity
$\alpha( \sigma \circ \sigma ) \subseteq (\alpha \sigma ) \alt h$,
where  $\sigma$ 
is a U-admissible relation, equivalently, a
U$_2$-admissible relation.

(2)
There is a function $f: \{0, 1,  \dots, h-1 \} \to \{ 1,2\}  $
such that one (and hence all) of the following equivalent conditions 
hold.
  \begin{enumerate}
    \item[(a)]  
$\mathcal V$ satisfies the identity
$\alpha( R_1 \circ R_2 ) \subseteq \alpha  R _{f(0)} \circ 
\alpha  R _{f(1)} \circ \dots \circ  \alpha  R _{f(h-1)}  $,
where $R_1$ and $R_2$ are reflexive and admissible relations. 
\item[(b)]
 $\mathcal V$ has $4$-ary terms
$s_i$, $i=0, \dots, h-1$  such that the following
 identities hold throughout $\mathcal V$. 
\begin{align*}
x =s_0(x,y,z;w_{f(0)}), \quad s_{h-1}(x,y,z;w'_{f(h-1)}) = z, \quad \text{ and } 
\\
\begin{aligned}[c]
s_{i}(x,y,z;w'_{f(i)})  & = s_{i+1}(x,y,z;w_{f(i+1)}) 
\\
x  & =  s_{i+1}(x,y,x;w'_{f(i)})
\end{aligned}
\quad \text{ for }   i < h-1, 
\end{align*}
 where $w_1$, $w_2$, $w'_1$ and $w'_2$
denote, respectively, the variables
$x$, $y$, $y$ and $z$.
\end{enumerate} 
In the  identities in (1) and (2)(a)
 above we can let $\alpha$ be
equivalently  a variable for  congruences or tolerances.
 \end{theorem} 

 \begin{proof}
That in each case we get an equivalent condition 
letting $\alpha$ be a  congruence or tolerance
is proved using the argument from \cite{CH}
recalled at the end of the proof of  Proposition \ref{12}.
We shall deal here with the simpler case when $\alpha$ 
is a congruence.

Suppose that $\mathcal V$ satisfies the identity
$\alpha( \sigma \circ \sigma ) \subseteq (\alpha \sigma ) \alt h$,
where  $\sigma$ 
is a
U$_2$-admissible relation.
Let us work in the free algebra $\mathbf F _{ \mathcal V } (3) $
in $\mathcal V$ generated by the three elements 
$x$, $y$ and  $z$.
Let $\alpha$
be the congruence generated by 
$(x,z)$
and let
 $R_1$ and $R_2$ be, respectively,
the smallest reflexive and admissible relations containing
 $(x,y)$, respectively, $(y,z)$.
Let $\sigma= R_1 \cup R_2 $.   
We have that 
$(x,z) \in \alpha ( \sigma \circ \sigma )$,
hence, by (1),
 $(x,z) \in (\alpha  \sigma )  \alt h$.
  Hence there are ternary terms 
$t_0, t_1, \dots, t_h$ 
such that 
$x=t_0(x,y,z)$,
$t_i(x,y,z) \mathrel { \alpha  \sigma  } t_{i+1}(x,y,z)$,
for   $i=0, \dots, h-1$ 
and $t_h(x,y,z) =z$.
The usual arguments 
show that the identities 
$t_i(x,y,x)=x$ 
should hold throughout $\mathcal V$.
Moreover, 
since $\sigma= R_1 \cup R_2 $, then, for every 
$i= 0, \dots, h-1$,
either $t_i(x,y,z) \mathrel { R_1 } t_{i+1}(x,y,z)$ 
or
$t_i(x,y,z) \mathrel { R_2 } t_{i+1}(x,y,z)$.
Let $f$ be such that  
$t_i(x,y,z) \mathrel { R_{f(i)}} t_{i+1}(x,y,z)$.
Arguing as in the proof of 
Proposition \ref{vr}, 
there are  
$4$-ary terms 
$s_i$ such that 
$t_i(x,y,z) =s_i(x,y,z;w_{f(i)})$
and 
 $s_{i}(x,y,z;w'_{f(i)}) = t_{i+1}(x,y,z) $,
for $i=0, \dots, h-1$
and where $w_1$, $w_2$, $w'_1$ and $w'_2$
denote, respectively, the variables
$x$, $y$, $y$ and $z$.   
Eliminating the $t_i$'s from the equations we have obtained,
we get exactly the identities
in (2)(b).

We have proved that if (1) holds
for U$_2$-admissible relations, then there exists some function $f$
such that (2)(b) holds.

It is by now standard to show that, given some function $f$,
(2)(a) and (2)(b) are equivalent.   Obviously,
if (a) holds, then   (a) holds in $\mathbf F _{ \mathcal V } ( 3 ) $;
 moreover,
the argument above can be easily
reformulated in order to show that if (a) holds
in $\mathbf F _{ \mathcal V } ( 3 ) $, then   (b) holds.
On the other hand, the proof that (b) implies (a) is similar to the 
final part of the proof 
of Proposition \ref{vr}.
If $a \mathrel \alpha c $ and
$a \mathrel { R_1} b \mathrel { R_2} c  $,
then
$(a,c) \in \alpha  R _{f(0)} \circ 
\alpha  R _{f(1)} \circ \dots \circ  \alpha  R _{f(h-1)}$ 
is witnessed by the elements 
$s_i(a,b,c;d _{f(i)})$, for $i<h-1$,
where $d_1=b$ and $d_2=c$.  
Indeed, if, say,  $f(i)=f(i+1)=1$, then 
$ s_i(a,b,c;b) = s_{i+1}(a,b,c;a)  \mathrel { R_1} s_{i+1}(a,b,c;b) $ 
and 
$s_i(a,b,c;a) \mathrel \alpha  s_i(a,b,a;a)  = a= s_{i+1}(a,b,a;a) 
\mathrel \alpha  s_{i+1}(a,b,c;a) $.   

It is now easy to show that (2)(a), for some $f$, implies the stronger 
form of (1) for U-admissible relations.
Indeed, if $\sigma= \bigcup_{g \in G} R _g$, 
$a \mathrel \alpha c $ and
$a \mathrel \sigma  b \mathrel \sigma  c  $,
then 
$ a \mathrel { R_{g_1}} b \mathrel { R _{g_2}} c $,
for some $g_1, g_2 \in G$, 
hence (2)(a) implies
$(a,c) \in \alpha  R _{g_{f(0)}} \circ 
\alpha  R _{g_{f(1)}} \circ \dots \circ  \alpha  R _{g_{f(h-1)}}
\subseteq (\alpha \sigma ) \alt h $.   
 \end{proof}

\begin{remark} \labbel{a}    
Let (a)$_f$ denote the identity given by  (a) in Theorem \ref{mal}
when applied to some specific function $f$.
It is probably an interesting problem to determine which implications
hold among  identities of the form 
(a)$_f$ and (a)$_{f'}$,
letting $f$ and $f'$ vary,
possibly with $h \neq h'$.  
 Of course, if 
$k$ is the number of
``variations''   
in the sequence $f(0), f(1), \dots, f(h-1)$, then,
 letting $R_1$ and $R_2$
be congruences in (a)$_f$,
we obtain 
$\alpha( \beta \circ \gamma ) \subseteq \alpha \beta \circ_{k+1} \alpha \gamma  $
if $f(0)= 1$,
and  
$\alpha( \beta \circ \gamma ) \subseteq \alpha \gamma  \circ_{k+1} \alpha \beta  $
if $f(0)= 2$.
By \cite{kkmm} and the last statement in Proposition \ref{12b},
if one of the above congruence identities holds, then
$\alpha(R \circ S) \subseteq \alpha R \circ _{k'} \alpha S $
holds, for some $k'$, but the known proofs, so far, 
provide a relatively large $k'$.    

As a related observation,
we know that, in general, adjacent identical relations do not absorb
in relation identities.
For example, it follows from equation (10) in \cite{baker} 
that Baker's variety (cf. Proposition \ref{baker})
satisfies 
$T (R \circ S \circ R) \subseteq TR \circ TS \circ TR \circ TR \circ TS \circ TR$
but fails to satisfy 
$T (R \circ S \circ R) \subseteq TR \circ TS \circ TR \circ TS \circ TR$.
Though there are probably similar examples in which the left-hand side has the form
$T (R \circ S)$ or 
$ \Theta  (R \circ S)$, at present we have none of them at hand.
\end{remark}

\begin{problem} \labbel{int}   
As pointed out by several authors, e.g.\ 
J{\'o}nsson \cite[p. 370]{cv} or Tschantz \cite{T},
it  is interesting to study  reflexive and admissible relations 
on some algebra, as well as the identities they satisfy  within a variety.
Here the identities can be constructed using
 the operations of intersection, converse, composition
and transitive closure.
The  results presented in this note suggest that it might be interesting to 
study also identities satisfied by $U$-admissible relations.
Notice that  the set of $U$-admissible relations
on some algebra is closed under the above-mentioned operations, as well as,
obviously, under set-theoretical union and under taking
admissible closure.
Hence it might be  interesting to study 
the structure 
$\mathscr U (\mathbf A)=
( U, \cap, \cup, \circ,  ^\smallsmile ,  \astt ,   \bar{\ } )$
associated to any algebra $\mathbf A$, 
where $U$ is the set of all
U-admissible relations of $\mathbf A$ and 
$\bar{\ }$ denotes the operation of taking the smallest 
reflexive and admissible relation containing the argument. 
\end{problem}

Obviously, the arguments from Propositions \ref{vr} and
Theorem \ref{mal} can be merged in order to obtain a characterization of the 
identity $\sigma( \tau \circ \tau ) \subseteq (\sigma \tau ) \alt h$,
for U-admissible relations.  

More generally, parts of the arguments in the proof of Proposition \ref{mal}
can be inserted into a  broader context, related to Problem \ref{int}. 

If  $\varepsilon( \sigma_1, \sigma _2, \dots ) \subseteq 
\varepsilon' ( \sigma_1, \sigma _2, \dots) $  is an inclusion depending on 
the variables $\sigma_1$, $ \sigma _2$, \dots,
we say that an inclusion
$ \delta ( R_{1,1}, R_{1,2}, \dots, R_{1,j_1}, R_{2,1}, \dots ) \subseteq 
\delta' ( R_{1,1}, R_{1,2}, \dots,  R_{1,j_1}, R_{2,1}, \dots ) $
is an \emph{expansion} of $\varepsilon \subseteq \varepsilon '$ 
if 
$j_1$ is the number of occurrences of $\sigma_1$ in 
$\varepsilon$ and similarly for the other variables, and    
$\delta \subseteq \delta '$ can be obtained in the following way.
First,
$\delta ( R_{1,1}, R_{1,2}, \dots, \allowbreak R_{2,1} \dots ) $ is obtained 
from $\varepsilon( \sigma_1, \sigma _2, \dots )$ by substituting
each occurrence of $\sigma_1$ for a \emph{distinct} variable
from the sequence  
  $R_{1,1}, R_{1,2}, \dots, R_{1,j_1}$ and similarly for 
the other variables.
This first step can be performed essentially in a unique way, 
modulo renamings.
Then we require that 
$\delta' ( R_{1,1}, R_{1,2}, \dots, R_{2,1} \allowbreak  \dots ) $ is obtained 
from $\varepsilon ' ( \sigma_1, \sigma _2, \dots )$ by substituting
every  occurrence of $\sigma_1$ for some variable
from the sequence  
  $R_{1,1}, R_{1,2}, \allowbreak \dots, R_{1,j_1}$, in any order and with the possibility 
of repetitions.
Similarly for 
$\sigma_2$ and for all the other variables.
This second step can be performed in many nonequivalent ways.
For example, for every $f$ as in Condition (2) in Proposition \ref{mal},
 the identity obtained in (2)(a) is one of the many possible
 expansions of the identity from
(1) in Proposition \ref{mal}.

\begin{proposition} \labbel{mall}
Suppose that  $\varepsilon( \sigma_1, \sigma _2, \dots ) \subseteq 
\varepsilon' ( \sigma_1, \sigma _2, \dots) $  is an inclusion
built using the  operations
$\cap$, $\circ$ and $ ^\smallsmile $,
where $\sigma_1$, $ \sigma _2$, \dots \ are intended
to be variables for 
U-admissible relations.
Then a variety $\mathcal V$
satisfies $\varepsilon( \sigma_1, \sigma _2, \dots ) \subseteq 
\varepsilon' ( \sigma_1, \sigma _2, \dots) $ if and only if 
$\mathcal V$ satisfies at least one  expansion
$ \delta ( R_{1,1}, R_{1,2}, \dots, \allowbreak R_{2,1}, \dots ) \subseteq 
\delta' ( R_{1,1}, R_{1,2}, \dots, R_{2,1}, \dots ) $
of $\varepsilon  \subseteq \varepsilon '$,
where  the $R _{i,j} $'s are  interpreted
as  admissible relations.
 \end{proposition}

\acknowledgement{We thank an anonymous referee for many
useful comments which helped improving the paper.
We thank the students of Tor Vergata
University for  stimulating discussions.}

\smallskip

{\scriptsize
The author 
 considers that it is highly  inappropriate, 
and strongly discourages, the use 
of indicators extracted from the following  list 
(even in aggregate forms in combination with similar lists)
 in decisions about individuals 
(job opportunities, career progressions etc.), attributions of funds
 and selections or evaluations of research projects.
\par
}

\end{document}